\documentclass[12pt,reqno]{amsart}
\usepackage{amscd,amssymb,amsmath,amsthm}
\usepackage{cite}

\newtheorem{thm}[subsection]{Theorem}
\newtheorem{lemma}[subsection]{Lemma}
\newtheorem{pro}[subsection]{Proposition}
\newtheorem{cor}[subsection]{Corollary}

\newtheorem{rk}[subsection]{Remark}

\numberwithin{equation}{section} 

\newcommand{\bea}{\begin{eqnarray}}
\newcommand{\eea}{\end{eqnarray}}

\newcommand{\Z}{\mathbb{Z}}
\newcommand{\Q}{\mathbb{Q}}

\def\bn{{\mathbb N}}

\def\bq{{\mathbb Q}}

\def\bz{{\mathbb Z}}

\def\a{\alpha}

\def\e{\epsilon}

\begin{document}

\title[$p$-adic monomial equations]
{$P$-adic monomial equations and their perturbations}

\author{Farrukh Mukhamedov}
\address{Farrukh Mukhamedov\\
 Department of Mathematical Sciences\\
College of Science, The United Arab Emirates University\\
P.O. Box, 15551, Al Ain\\
Abu Dhabi, UAE} \email{{\tt far75m@gmail.com} {\tt
farrukh.m@uaeu.ac.ae}}

\author{ Otabek Khakimov}
\address{ Otabek Khakimov\\
 Department of Mathematical Sciences\\
College of Science, The United Arab Emirates University\\
P.O. Box, 15551, Al Ain\\
Abu Dhabi, UAE} \email{{\tt hakimovo@mail.ru}}

\begin{abstract}
In this paper, we describe the set of all solutions
of monomial equation $x^k=a$ over $\mathbb Q_p$. Moreover, as
an application of the result, we study several perturbations of the considered equation over $p$-adic field.
monomial equation, p-adic numbers, rational function; \end{abstract}

\maketitle

\section{Introduction}

At present, $p$-adic analysis is a rapidly developing trend in mathematics \cite{KKZ}. Numerous applications
of $p$-adic numbers have resulted in the theory of $p$-adic differential equations, $p$-adic probability theory,
$p$-adic mathematical physics, etc. (see  \cite{ALB,AK,Drag,Drag2,Khr,VVZ}).
It is known \cite{Borevich} that $p$-adic numbers are also closely connected with Diophantine equations, i.e. in finding all solutions of a system of polynomial equations or to give a bound for the number of
solutions over the field of $p$-adic numbers $\bq_p$.  We notice that, in general, the same Diophantine problem
may have different solutions from the field of $p$-adic numbers to the field of real numbers
because of the different topological structures \cite{MOSM13}.

One of the simplest of these equations is an equation of the form $x^q = a$
over $\bq_p$, where $q\in\bn$ and $a\in\bq_p$. A criterion for the solvability of this equation from the standpoint of
algebraic number theory was given in \cite{FMMS_JNT}. In \cite{COR} the mentioned criterion has been applied to classification problems of Leibnitz algebras over $\mathbb Q_p$.
Unfortunately, in both papers \cite{COR,FMMS_JNT} the authors have not interested in a number of solutions of the monomial
equation. This question, it is very important to know locations of solutions  when one studies the fixed point problems and dynamical behavior of $p$-adic rational functions \cite{ReportMath,MFKhO3}.
Therefore, in the present paper, we describe the set of solutions of the monomial equation and their locations.  Moreover,
we will provide its applications for solving several nonlinear equations over $\bq_p$. Namely, we consider a perturbation of the monomial equation, i.e. $F(x)=x^n+\e f(x)$, where $f:\bz_p\to\bz_p$ is a $p^\a$-Lipschitz function. Under some conditions,
we are able to find locations of solutions of the equation $F(x)=0$ by means of a generalized Hensel's Lemma \cite{YurKhren} with certain modification. This result also can be considered as a natural application of the generalized Helsel's Lemma \cite{YurKhren,YurKhren1} to the
perturbed monomial systems, but in our context, we establish a number of its solutions, and their locations as well.  Furthermore, a number of fixed points of Bethe-Potts type (see \cite{M14,MFKhO3}) rational functions is obtained.

\section{Preliminaries}
 In this section, we recall some definitions related to the $p$-adic analysis and we introduce
the necessary notations.

Let $\Q$ be the field of rational numbers. For a fixed prime number
$p$, every rational number $x\ne 0$ can be represented in the form
$x = p^r{n\over m}$, where $r, n\in \Z$, $m$ is a positive integer,
and $n$ and $m$ are relatively prime with $p$: $(p, n) = 1$, $(p, m)
= 1$. The $p$-adic norm of $x$ is given by
$$|x|_p=\left\{\begin{array}{ll}
p^{-r}\ \ \mbox{for} \ \ x\ne 0\\
0\ \ \mbox{for} \ \ x = 0.
\end{array}\right.
$$
This norm is non-Archimedean  and satisfies the so called {\it strong
triangle inequality}
$$|x+y|_p\leq \max\{|x|_p,|y|_p\}.$$

The completion of $\Q$ with respect to the $p$-adic norm defines
the $p$-adic field
 $\Q_p$. We point out that $\mathbb Q_p$ is not an ordered field \cite{sch}.

Any $p$-adic number $x\ne 0$ can be uniquely represented
in the canonical form
\begin{equation}\label{ek}
x = p^{\gamma(x)}(x_0+x_1p+x_2p^2+\dots),
\end{equation}
where $\gamma(x)\in \Z$ and the integers $x_j$ satisfy: $x_0 > 0$,
$0\leq x_j \leq p - 1$. In this case, $|x|_p =
p^{-\gamma(x)}$.

In what follows, to simplify our calculations, we are going to introduce new symbols "O" and "o" (Roughly speaking, these symbols replace the notation $\equiv \ (mod \ p^k)$ without noticing about power of $k$).
 Namely, for a given $p$-adic number $x$ by $O[x]$ we mean a $p$-adic number with the norm $p^{-\gamma(x)}$, i.e. $|x|_p=|O(x)|_p$. By
 $o[x]$, we mean a $p$-adic number with a norm strictly less than $p^{-\gamma(x)}$, i.e. $|o(x)|_p<|x|_p$.
 For instance,
 if $x=1-p+p^2$, we can write $O[1]=x$, $o[1]=x-1$ or $o[p]=x-1+p$.
Therefore, the symbols $O[\cdot]$ and $o[\cdot]$ make our work easier when we need
to calculate the $p$-adic norm of $p$-adic numbers.
It is easy to see that $y=O[x]$ if and only if $x=O[y]$.


We give some basic properties of $O[\cdot]$ and $o[\cdot]$, which will be used later on.

\begin{lemma} Let $x,y\in\mathbb Q_p$. Then the following statements hold:\\
\begin{enumerate}
\item[$1^\circ$.] $O[x]O[y]=O[xy]$;\\
\item[$2^\circ$.] $xO[y]=O[y]x=O[xy]$;\\
\item[$3^\circ$.] $O[x]o[y]=o[xy]$;\\
\item[$4^\circ$.] $o[x]o[y]=o[xy]$;\\
\item[$5^\circ$.] $xo[y]=o[y]x=o[xy]$;\\
\item[$6^\circ$.] $\frac{O[x]}{O[y]}=O\left[\frac{x}{y}\right]$, if $y\neq0$;\\
\item[$7^\circ$.] $\frac{o[x]}{O[y]}=o\left[\frac{x}{y}\right]$, if $y\neq0$.
\end{enumerate}
\end{lemma}
\begin{proof}
It is enough to  prove assumptions $1^\circ$ and $2^\circ$, since the others can be proceeded by the same manner.
Due to $x=O[x]$ for all $x\in\mathbb Q_p$, one has $xy=O[x]O[y]$. Noting $xy=O[xy]$,
we obtain $O[x]O[y]=O[xy]$. Similarly, we find $xO[y]=O[x]O[y]=O[xy]$.
\end{proof}

For each $a\in \bq_p$ and $r>0$, we denote $$ B_r(a)=\{x\in \bq_p :
|x-a|_p< r\}.$$ We recall that $\mathbb{Z}_p=\{x\in \Q_p:
|x|_p\leq 1\}$ and $\mathbb Z_p^*=\{x\in\mathbb Q_p: |x|_p=1\}$ are the set of all \textit{$p$-adic integers}
and {\it $p$-adic units}, respectively.

The following result is known as the Hensel's lemma
\begin{lemma}\cite{Borevich,Ko}
Let $F(x)$ be a polynomial whose coefficients are $p$-adic integers.
Let $x^*$ be a $p$-adic integer such that for some $i\geq0$
one has
$$
F(x^*)\equiv0(\operatorname{mod }p^{2i+1}),\ \ \ F'(x^*)\equiv0(\operatorname{mod }p^{i}),\ \ \ F'(x^*)\not\equiv0(\operatorname{mod }p^{i+1}).
$$
Then $F(x)$ has a $p$-adic integer root $x_*$ such that $x_*\equiv x^*(\operatorname{mod }p^{i+1})$.
\end{lemma}

\begin{rk} In \cite{YurKhren,YurKhren1} it was given a generalization of Hensel Lemma for $1$-Lipschitz functions using van der Put decomposition. We notice that other applications of the
van der Put basis have been considered in \cite{AKY}.
\end{rk}

Recall that the $p$-adic exponential is defined by
$$
\exp_p(x)=\sum_{n=0}^{\infty}\frac{x^n}{n!},
$$
which converges for every $x\in B_{p^{-1/(p-1)}}(0)$.
Denote
$$
\mathcal E_p=\left\{x\in\mathbb Q_p: |x-1|_p<p^{-1/(p-1)}\right\}.
$$
This set is the range of the $p$-adic exponential function \cite{Ko,sch}. In the sequel, the following well
known fact will be frequently used without noticing.

\begin{lemma}\cite{sch}\label{epproperty}
Let $p\geq3$. The set $\mathcal E_p$ has the following properties:\\
$(a)$ $\mathcal E_p$ is a group under multiplication;\\
$(b)$ $|a-b|_p<1$ for all $a,b\in\mathcal E_p$;\\
$(c)$ if $a,b\in\mathcal E_p$ then
$|a+b|_p=1$.\\
$(d)$ if $a\in\mathcal E_p$, then
there is an element $h\in B_{p^{-1/(p-1)}}(0)$ such that
$a=\exp_p(h)$.
\end{lemma}

In what follows, we need  some auxiliary facts.

\begin{lemma}\cite{MFKhO3}\label{alpbetgam}
Let $k\geq2$ and $p\geq3$. Then for any $\alpha,\beta\in\mathcal E_p$ there exists
$\gamma\in\mathcal E_p$ such that
\begin{equation}\label{abg}
\sum_{j=0}^{k-1}\alpha^{k-j-1}\beta^j=k\gamma
\end{equation}
\end{lemma}

It is well known that one can decompose $\mathbb Z_p^*=\bigcup\limits_{j=1}^{p-1}B_1(j)$.  Hence,  as a corollary of Lemma \ref{alpbetgam}
we can the following fact.

\begin{lemma}\label{coralb}
Let $p\geq3$ and $k\geq2$. If $x,y\in B_1(j)$ for some $j\in\{1,2,\dots,p-1\}$ then one has
$$
x^k-y^k=k(x-y)y^{k-1}+o\left[k(x-y)\right].
$$
\end{lemma}

\begin{proof} Let $x,y\in B_1(j)$, then $\frac{x}{y}\in\mathcal E_p$.
Due to Lemma \ref{alpbetgam} there exists $\gamma\in\mathcal E_p$ such that
$$
\sum_{j=0}^{k-1}\left(\frac{x}{y}\right)^{k-j-1}=k+k(\gamma-1).
$$

Hence, using the last expression, we obtain
\begin{eqnarray*}
x^k-y^k&=&y^k\left(\frac{x}{y}-1\right)\sum_{j=0}^{k-1}\left(\frac{x}{y}\right)^{k-j-1}\\
&=&k(x-y)y^{k-1}+k(x-y)y^{k-1}(\gamma-1)\\
&=&k(x-y)y^{k-1}+O[k(x-y)]o[1]\\
&=&k(x-y)y^{k-1}+o[k(x-y)]
\end{eqnarray*}
which completes the proof.
\end{proof}

\begin{cor}\label{coralb1}
Let $p\geq3$ and $k\geq2$. Then for any $x\in\mathbb Z_p^*$ with canonical representation
$$
x=x_0+x_1p+x_2p^2+\dots
$$
one has
$$
x^k-x_0^k=o[k].
$$
\end{cor}

\section{The equation $x^k=a$}

In this section, we are going to describe locations of  all solutions
of the equation
\begin{equation}\label{asequat}
x^k = a, \ \ k\in\bn,\  a\in\bq_p
\end{equation}
 in $\mathbb Q_p$. Let us first notice
that the equation \eqref{asequat} can be considered over $\mathbb Z_p^*$.
Indeed, any nonzero $p$-adic number $x$ has a unique representation of the form $x=\frac{x^*}{|x|_p}$, where $x^*\in\mathbb Z_p^*$.
After
substituting the forms $x=\frac{x^*}{|x|_p}$, $a=\frac{a^*}{|a|_p}$ into \eqref{asequat}, we can get that
$\left(\frac{x_*}{|x|_p}\right)^k=\frac{a_*}{|a|_p}$.
This means that Eq. \eqref{asequat} has a solution in $\mathbb Q_p$ whenever $a\in\mathbb Q_p$ if and only if $|a|_p=p^{kl}$ for some
$l\in\mathbb Z$ and the equation $x_*^k=a_*$ has a solution in $\mathbb Z_p^*$.
Hence, in what follows, we may always assume that $a\in\mathbb Z_p^*$.

Recall that $b\in\mathbb Z$ is called an {\it $m$-th power residue modulo} $p$ if the congruence equation
$x^m\equiv b(\operatorname{mod} p)$ has a solution in $\mathbb Z$. In  the sequel, we always
assume that $p\geq3$.
It is well-known (see \cite{Rosen}) that $b\in\mathbb Z$ is
an $m$-th power residue modulo $p$ if and only if
$b^{\frac{p-1}{d}}\equiv1(\operatorname{mod }p)$, where $d=(m, p-1)$.

\begin{thm}\cite{FMMS_JNT}\label{FM}
Let $p\geq3$ and $k=mp^s$, where $(p,m)=1$, $s\geq0$.
Assume that $a\in\mathbb Q_p$ has the following canonical representation
$$
a=a_0+a_1p+a_2p^2+\dots,\ \ \ \ a_0\in\{1,2,\dots,p-1\}.
$$
Then the equation \eqref{asequat}
has a solution on $\mathbb Z_p^*$ if and only if
\begin{enumerate}
\item[(i)] $a_0$ is $m$-th power residue modulo $p$;
\item[(ii)] $a_0^{p^s}\equiv a(\operatorname{mod }p^{s+1})$.
\end{enumerate}
\end{thm}
Although Theorem \ref{FM} presents as a solvability criteria for Eq. \eqref{asequat},
but it is not useful in finding a number of its solutions and their locations. In this section, we are going to resolve the mentioned issues.

For a given $a\in\mathbb Z_p^*$,  let us denote
$$
Sol_p(x^k-a)=\left\{\xi\in\mathbb F_p: \xi^k-a=o[1]\right\}
$$
 and
$$
\kappa_p=card\left(Sol_p(x^k-a)\right),
$$
where $\mathbb F_p$ is a ring of integers modulo $p$ and $card(A)$ stands for the cardinality of a set $A$. We notice that
$0\leq\kappa_p\leq k$. We observe that the condition $Sol_p(x^k-a)\neq\emptyset$ is equivalent to
$a_0$ is $k$-th power residue modulo $p$.

\begin{rk}
According to Theorem \ref{FM}, a necessity condition for the solvability
of \eqref{asequat} is $Sol_p(x^k-a)\neq\emptyset$. Namely,
let $\overline{x}\in\mathbb Z_p^*$, with canonical representation
$$
\overline{x}=\overline{x}_0+\overline{x}_1p+\overline{x}_2p^2+\dots
$$
If $\overline{x}$ is a solution of \eqref{asequat} then
$\overline{x}_0\in Sol_p(x^k-a)$. Hence, for any $\xi\in\mathbb F_p\setminus Sol_p(x^k-a)$
the equation \eqref{asequat} has no solution in $B_1(\xi)$.
\end{rk}

Now we provide a main result of this section which extends Theorem \ref{FM}.

\begin{thm}\label{thm_main1}
Let $p\geq3$ and $k=mp^s$, where $(p,m)=1$, $s\geq0$. Assume that $a\in\mathbb Z_p^*$ and has the following canonical form
$$
a=a_0+a_1p+a_2p^2+\dots,
$$
and
$Sol_p(x^k-a)\neq\emptyset$. Then the followings statements are equivalent:
\begin{enumerate}
\item[(i)] the equation \eqref{asequat} has a solution;
\item[(ii)] $a=a_0^{p^s}+o\left[p^s\right]$;
\item[(iii)] for any $\xi\in Sol_p(x^k-a)$ Eq. \eqref{asequat} has a unique solution in $B_1(\xi)$.
\end{enumerate}
\end{thm}

\begin{proof} The implication (iii)$\Rightarrow$(i) is obvious. So, we will prove only (i)$\Rightarrow$(ii)
and (ii)$\Rightarrow$(iii).

(i)$\Rightarrow$(ii) Let us suppose that (i) holds, i.e. \eqref{asequat} has a solution $\overline{x}$. Keeping in mind
$k=mp^s$, due to
Corollary \ref{coralb1} we have
\begin{equation}\label{111}
\overline{x}^k-\overline{x}_0^k=o[p^s],
\end{equation}
which implies
$a=\overline{x}_0^k+o[1]$. Then from $a=a_0+o[1]$ one gets
\begin{equation}\label{11}
\overline{x}_0^k=a_0+o[1].
\end{equation}
On the other hand, we have $\overline{x}_0^{p^s}=\overline{x}_0+o[1]$ (It follows from Fermat's little theorem).
Then according to Lemma \ref{coralb}
one finds $\overline{x}_0^{k}=\overline{x}_0^m+o[1]$. The last one together with \eqref{11}
yields
\begin{equation}\label{12}
\overline{x}_0^m=a_0+o[1].
\end{equation}
Hence, again using Lemma \ref{coralb}, from \eqref{12} we obtain
\begin{eqnarray}\label{13}
\overline{x}_0^k-a_0^{p^s}=\left(\overline{x}_0^m\right)^{p^s}-a_0^{p^s}&=&p^{s}(\overline{x}_0^m-a_0)+o\left[p^{s}(\overline{x}_0^m-a_0)\right]\nonumber\\
&=&p^so[1]+o\left[p^s\right]\nonumber\\
&=&o\left[p^s\right]
\end{eqnarray}
From \eqref{111} and \eqref{13} we immediately get (ii).

(ii)$\Rightarrow$(iii) Now, we assume that (ii) holds. Take an arbitrary $\xi\in Sol_p(x^k-a)$ and
define a sequence $\{x_{(n)}\}_{n\geq1}$ as follows:
\begin{equation}\label{seq}
x_{(1)}=\xi,\ \ \ x_{(n)}=x_{(n-1)}+\frac{a-x_{(n-1)}^k}{kx_{(n-1)}^{k-1}},\ \ n\geq2.
\end{equation}
First by induction, we show that $x_{(n)}\in B_1(\xi)$ for any $n\geq1$. Since
$x_{(1)}\in B_1(\xi)$ we may assume that $x_{(n)}\in B_1(\xi)$ for some $n\geq1$.
Thanks to Corollary \ref{coralb1} we obtain
\begin{equation}\label{16}
x_{(n)}^k-\xi^k=o[k].
\end{equation}
On the other hand, we have $\xi^k=a_0+o[1]$, which together with $\xi^{p^s}=\xi+o[1]$ yields
$\xi^m=a_0+o[1]$. Due to Lemma \ref{coralb} we infer
\begin{eqnarray}\label{17}
\xi^k-a_0^{p^s}&=&p^s(\xi^m-a_0)+o\left[p^s(\xi^m-a_0)\right]\nonumber\\
&=&p^so[1]+o\left[p^s\right]\nonumber\\
&=&o[p^s]
\end{eqnarray}
From \eqref{16} and \eqref{17}, noting $k=O[p^s]$ one has
\begin{equation}\label{18}
x_{(n)}^k-a_0^{p^s}=o[k].
\end{equation}
Putting (ii) and \eqref{18} into \eqref{seq} we obtain
$$
x_{(n+1)}=x_{(n)}+o[1],
$$
which implies $x_{(n+1)}\in B_1(x_{(n)})$. The last one together with $x_{(n)}\in B_1(\xi)$ yields that
$x_{(n+1)}\in B_1(\xi)$. Hence, $x_{(n)}\in B_1(\xi)$ for any $n\geq1$.

Furthermore, using Lemma \ref{coralb}, from \eqref{seq} one can find
$$
x_{(n)}^k-x_{(n-1)}^k=a-x_{(n-1)}^k+o\left[a-x_{(n-1)}^k\right],
$$
which implies
$$
x_{(n)}^k-a=o\left[x_{(n-1)}^k-a\right].
$$
The last one together with $x_{(1)}^k-a=o[1]$ yields
$$
x_{(n)}^k-a=o\left[p^{n-1}\right].
$$
So, we can find a number $n_0\geq1$ such that $x_{(n_0)}^k-a=o[k^2]$. Let us consider a
polynomial $F(x)=x^k-a$. It is easy to check that $|F'(x_{n_0})|_p=|k|_p$ and
$|F(x_{(n_0)})|_p<|k^2|_p$. Thanks to Hensel's Lemma, $F$ has a root $x_\xi$
such that $|x_\xi-x_{(n_0)}|_p<1$, which implies $x_\xi\in B_1(\xi)$.

Let us suppose that $x$ and $y$ are solutions of \eqref{asequat} belong to $B_1(\xi)$.
Then according to Lemma \ref{coralb} we obtain
$$
|x^k-y^k|_p=|k(x-y)|_p.
$$
The last one together with $x^k-y^k=0$ imply $x=y$.

Thus we have proved that for any $\xi\in Sol_p(x^k-a)$ the equation \eqref{asequat} has a
unique solution in $B_1(\xi)$.
This completes the proof.
\end{proof}

\begin{cor}\label{cor1}
Let $p\geq3$ and $k=mp^s$, where $(p,m)=1$, $s\geq0$. Then for any $a\in\mathbb Z_p^*$
with canonical representation
$$
a=a_0+a_1p+a_2p^2+\dots,
$$
there exists $x_*\in\mathbb Z_p^*$ such that
$x_*^k=a+o[k^2]$ iff the following statements hold:
\begin{enumerate}
\item[(i)] $Sol_p(x^k-a)\neq\emptyset$;
\item[(ii)] $a=a_0^{p^s}+o\left[p^s\right]$.
\end{enumerate}
\end{cor}

\begin{cor}\label{coryaxshi}
Let $p\geq3$ and $a\in\mathcal E_p$. Then the following statements hold:
\begin{enumerate}
\item[(i)]
if $|k|_p\leq|a-1|_p$ then \eqref{asequat} has no solution.
\item[(ii)]
if $|k|_p>|a-1|_p$ then \eqref{asequat} has exactly $\kappa_p$ solutions $x_{\xi_i}\in B_1(\xi_i)$,
$i\in\{1,2,\dots,\kappa_p\}$,
where $\xi_i\in Sol_p(x^k-a)$.
\end{enumerate}
\end{cor}

\begin{proof} We notice that for any $k\in\mathbb N$ and $a\in\mathcal E_p$
the set $Sol_p(x^k-a)$ is not empty.

(i) Since $a_0=1$ by definition $|a-a_0|_p\geq|k|_p$ implies that $a-a_0\neq o[k]$. Then
due to Theorem \ref{thm_main1} we infer that the equation \eqref{asequat} has no solution.

(ii) Suppose that $|k|_p>|a-1|_p$. This yields that $a-a_0=o[k]$. Then according to Theorem \ref{thm_main1} equation \eqref{asequat}
has exactly $\kappa_p$ solutions.
\end{proof}

\begin{rk}
Thanks to Corollary \ref{coryaxshi}, every $a\in\mathcal E_p$ with $|a-1|_p<|k|_p$
has a single $k$-th root on $\mathcal E_p$, which is called {\it the principal $k$-th root}
and it is denoted by $\sqrt[k]{a}$. In what follows, when we write $\sqrt[k]{a}$ for given
$a\in\mathcal E_p$
we always mean the principal $k$-th root of $a$.
\end{rk}

\section{Applications}

In this section, we provide certain applications of Theorem \ref{thm_main1} in solving nonlinear  equations over $p$-adic field.

\subsection{Roots of 1-Lipschitz functions}

Recall that
$f:\mathbb Z_p\to\mathbb Z_p$ is {\it $L$-Lipschitz function} if
\begin{equation}\label{Lip}
|f(x)-f(y)|_p\leq L|x-y|_p,\ \ \ \forall x,y\in\mathbb Z_p.
\end{equation}
We notice that the discretness of $p$-adic norm's values yields $L=p^\alpha$ for some $\alpha\in\mathbb Z$.
So, the condition \eqref{Lip} is equivalent to: $x-y=o[p^k]$ implies
$f(x)-f(y)=o[p^{k-\alpha}]$ for all $k\geq\alpha$.

For a given $k=mp^s$ with $(m,p)=1$, $s\geq0$ we consider the following function
\begin{equation}\label{Func}
F(x)=x^k-a+\varepsilon f(x),
\end{equation}
where $a\in\mathbb Z_p^*$, $|\varepsilon|_p<|k^2|_p$ and $f:\mathbb Z_p\to\mathbb Z_p$ is a $L$-Lipschitz
function with $L\leq p^s$.
One can see that $F(\mathbb Z_p)\subset\mathbb Z_p$. Moreover, $F$ is a 1-Lipschitz
function. In \cite{YurKhren} the authors found a criterion for the existence of a root of
1-Lipschitz
functions. Moreover, they were able to prove the analogue of Hensel's Lemma for
1-Lipschitz functions. In this section, we will show
locations of roots of the equation \eqref{Func} which was not studied in \cite{YurKhren,YurKhren1}.

\begin{thm}\label{thm1}
Let $p\geq3$ and $k=mp^s$ with $(p,m)=1$ and $s\geq0$. Assume that
$F$ is a function given by \eqref{Func}. Then the following statements hold:
\begin{enumerate}
\item[(F1)] if $Sol_p\left(x^k-a\right)=\emptyset$ then $F$ has no root in $\mathbb Z_p$;
\item[(F2)] if $Sol_p\left(x^k-a\right)\neq\emptyset$ then the followings are equivalent:
\begin{enumerate}
\item[(i)] $F$ has a root in $\mathbb Z_p$;
\item[(ii)] $a=a_0^{p^s}+o\left[p^s\right]$;
\item[(iii)] $F$ has exactly $\kappa_p$ roots in $\mathbb Z_p$: $x_{\xi_i}\in B_1(\xi_i)$, $\xi\in Sol_p(x^k-a)$, $i=\overline{1,\kappa_p}$.
\end{enumerate}
\end{enumerate}
\end{thm}

\begin{proof}
$(F1)$ Let $Sol_p(x^k-a)=\emptyset$. Pick any $x\in\mathbb Z_p$. Then we have $x^k-a=O[1]$.
Noting $\varepsilon f(x)=o[1]$
we infer $F(x)=O[1]$, which implies $F(x)\neq0$. Arbitrariness of $x\in\mathbb Z_p$
shows that $F$ has no root in $\mathbb Z_p$.

$(F2)$ Suppose that $Sol_p(x^k-a)\neq\emptyset$. It is enough to show the following implications
(i)$\Rightarrow$(ii)$\Rightarrow$(iii).

(i)$\Rightarrow$(ii). We assume that
$F$ has a root $x_*\in\mathbb Z_p$. Since $\varepsilon f(x)=o[k^2]$ for any $x\in\mathbb Z_p$,
we infer that $x_*^k-a=o[k^2]$. Then due to Corollary \ref{cor1},  (ii) holds.

(ii)$\Rightarrow$(iii). Take any $\xi\in Sol_p(x^k-a)$ and show that
the function $F$ has a unique root in $B_1(\xi)$. According to Theorem
\ref{thm_main1} there exists $x_{(1)}\in B_1(\xi)$ such that $x_{(1)}^k-a=o[k^2]$. Hence,
$\varepsilon f(x_{(1)})=o[p^{2s}]$ yields $F(x_{(1)})=o[p^{2s}]$.
Assume that for any $n\geq1$ we already have numbers $x_{(1)},x_{(2)},\dots,x_{(n)}\in B_1(\xi)$
such that
$$
F(x_{(m)})=o[p^{2s+m-1}],\ \ \ \forall m\leq n.
$$
Now we show the existence an $x_{(n+1)}\in B_1(\xi)$ satisfying $F(x_{(n+1)})=o[p^{2s+n}]$.
Let us denote
\begin{equation}\label{seq4}
x_{(n+1)}=x_{(n)}+\alpha_np^{s+n}
\end{equation}
where $\alpha_n\in\{0,1,\dots,p-1\}$. We need to choose $\alpha_n$ in order to $F(x_{(n+1)})=o[p^{2s+n}]$.

One can check that there exists only one $\alpha_n\in\{0,1,\dots,p-1\}$ which satisfies
\begin{equation}\label{alp}
F(x_{(n)})+k\alpha_np^{s+n}x_{(n)}^{k-1}=o[p^{2s+n}].
\end{equation}
By substituting $\alpha_n$ into \eqref{seq4}, we need to check $F(x_{(n+1)})=o[p^{2s+n}]$.
We first notice that the condition $x_{(n)}\in B_1(\xi)$ implies
$x_{(n+1)}\in B_1(\xi)$. Hence, using Lemma \ref{coralb} from \eqref{seq4}
we obtain
\begin{equation}\label{1u2}
x_{(n+1)}^k-x_{(n)}^k=k\alpha_np^{s+n}x_{(n)}^{k-1}+o[p^{2s+n}]
\end{equation}
Since $f$ is $L$-Lipschitz function with $L\leq p^s$ and noting $\varepsilon=o[p^{2s}]$ one gets
\begin{equation}\label{1u1}
\varepsilon f(x_{(n+1)})=\varepsilon f(x_{(n)})+o[p^{2s+n}].
\end{equation}
From \eqref{alp}-\eqref{1u1} we find
\begin{eqnarray*}
F(x_{(n+1)})&=&x_{(n+1)}^k-a+\varepsilon f(x_{(n+1)})\\
&=&x_{(n)}^k-a+\varepsilon f(x_{(n)})+k\alpha_np^{s+n}x_{(n)}^{k-1}+o[p^{2s+n}]\\
&=&F(x_{(n)})+k\alpha_np^{s+n}x_{(n)}^{k-1}+o[p^{2s+n}]\\
&=&o[p^{2s+n}].
\end{eqnarray*}
So, we have proved that for the sequence $\{x_{(n)}\}_{n\geq1}$ given by \eqref{seq4}
one has $F(x_{(n)})=o[p^{2s+n-1}]$. The sequence \eqref{seq4} is Cauchy and $B_1(\xi)$ is closed
in $\mathbb Q_p$, which yields the existence of $x_\xi\in B_1(\xi)$ such that
$F(x_{\xi})=o[p^{2s+n}]$ for all $n\geq1$. This means that $F(x_{\xi})=0$.

To establish that $F$ has a unique root in $B_1(\xi)$, we suppose the contrary. Namely, let  $x_\xi, y_\xi\in B_1(\xi)$
be two roots of $F$. Then, we obtain
$$
\varepsilon f(x_\xi)-\varepsilon f(y_\xi)=o[p^s(x_\xi-y_\xi)].
$$
Keeping in mind that fact, due to Lemma \ref{coralb} one has
\begin{eqnarray*}
F(x_\xi)-F(y_\xi)&=&x_\xi^k-y_\xi^k+o[p^s(x_\xi-y_\xi)]\\
&=&k(x_\xi-y_\xi)y_\xi^{k-1}+o[p^s(x_\xi-y_\xi)].
\end{eqnarray*}
which is equivalent to $\left|F(x_\xi)-F(y_\xi)\right|_p=|k(x_\xi-y_\xi)|_p$. Hence, we infer $x_\xi=y_\xi$.

Consequently, we have shown that for any $\xi\in Sol_p(x^k-a)$ the function $F$ has a unique root in
$B_1(\xi)$. From the arbitrariness of $\xi$ and due to $B_1(\xi)\cap B_1(\eta)=\emptyset$ for any $\eta\in Sol_p(x^k-a)$ with
$\xi\neq\eta$ we conclude that $F$ has exactly
$\kappa_p$ roots. This completes the proof.
\end{proof}

\begin{rk} We point out that the Hensel's Lemma \cite{Ko} reduces finding of roots (of polynomial equation) to finding roots of some iterative procedure
in each step of which, one needs to solve some linear congruence equation. In the above given procedure, it has been used a nonlinear congruence equation in each step of iteration, which
means we have implicitly used a generalized version of Hensel's Lemma with adoption to our equation.

\end{rk}

\subsection{Roots of Polynomials}

Let us consider the following polynomial
\begin{equation}\label{P(x)}
P_k(x)=x^k+\sum_{j=1}^{k}a^{(j)}x^{k-j},\ \ \ a^{(j)}\in\mathbb Z_p,\ j\in\{1,2,\dots,k\},
\end{equation}
where $a^{(j)}\in\mathbb Z_p$, $1\leq j\leq k$.
\begin{thm}\label{thm_poly}
Let $p\geq3$ and $k=mp^s$ with $(m,p)=1$, $s\geq0$. Assume that $P_k$ be a polynomial given by \eqref{P(x)}
such that
$a^{(k)}\in\mathbb Z_p^*$ and $\max\limits_{1\leq j<k}\left\{\left|a^{(j)}\right|_p\right\}<p^{-2s}$. Then $P_k$ has
a root iff the following statements hold:
\begin{enumerate}
\item[(P1)] $Sol_p\left(x^k+a^{(k)}\right)\neq\emptyset$;
\item[(P2)] $a^{(k)}=\left(a^{(k)}_0\right)^{p^s}+o[p^s]$.
\end{enumerate}
Moreover, under conditions $(P1)$,$(P2)$ the polynomial $P_k$ has exactly $\kappa_p$
solutions $x_{\xi_i}\in B_1(\xi_i)$, $i\in\{1,2,\dots,\kappa_p\}$, where
$\xi_i\in Sol_p\left(x^k+a^{(k)}\right)$ and $\kappa_p$ is a cardinality of $Sol_p\left(x^k+a^{(k)}\right)$.
\end{thm}
\begin{proof} It is easy to check that $|P_k(x)|_p\leq1$ if and only if $|x|_p\leq1$.
So, we infer that $P_k$ has no root in $\mathbb Q_p\setminus\mathbb Z_p$.

Due to $P_k(\mathbb Z_p)\subset\mathbb Z_p$ and Theorem \ref{thm1},
it is enough to establish that $P_k$ can be represented as \eqref{Func}.

Without lost of generality, we may assume that $a^{(j)}\neq0$ for some $j<k$.
Let us define
$$
f(x)=\sum_{j=1}^{k-1}\frac{a^{(j)}}{\varepsilon}x^{k-j},\ \ \
\mbox{where}\ \ \
\varepsilon=\frac{1}{\max\limits_{1\leq j<k}\{|a^{(j)}|_p\}}.
$$
Then $P_k$ can be written as follows:
\begin{equation}\label{P_k}
P_k(x)=x^k+a^{(k)}+\varepsilon f(x),
\end{equation}

We will show $|\varepsilon|_p<|k^2|_p$ and $f$ is $L$-Lipschitz function with $L\leq p^s$.

For any non zero $p$-adic number $y$ it holds $\left||y|_p\right|_p=\frac{1}{|y|_p}$. Using this fact
we can easily find $|\varepsilon|_p<|k^2|_p$. From $\frac{a^{(j)}}{\varepsilon}\in\mathbb Z_p$,
due to the strong triangle inequality we obtain $f(\mathbb Z_p)\subset\mathbb Z_p$.
On the other hand, $f$ is analytic and $|f'(x)|_p<1$ for all $x\in\mathbb Z_p$.
So, noting $f(\mathbb Z_p)\subset\mathbb Z_p$ we infer that $f$ is a $L$-Lipschitz function with $L<1$.
From $s\geq0$ we obtain $L<p^s$. This completes the proof.
\end{proof}

\subsection{Fixed point of rational functions}

In this subsection, we consider a function $f_{b,c,d}:\mathbb Q_p\to\mathbb Q_p$ given by
\begin{equation}\label{f_{b,c,d}}
f_{b,c,d}(x)=\left(\frac{bx-c}{x-d}\right)^k,\ \ \ b,c,d\in\mathcal E_p,\ \ c\neq bd.
\end{equation}
We notice that this type of functions are related to Bethe-Potts mappings associated with certain models of statistical mechanics \cite{ReportMath,M14}. To establish chaotic behavior of this kind of function, it is important to
describe location and a number of their fixed points. In this subsection, we apply Theorem \ref{thm_poly} to get the desired results.

First,  we recall that a point $x\in \bq_p$ is a \textit{fixed point} of $f_{b,c,d}$ if  $f_{b,c,d}(x)=x$.
By $Fix(f_{b,c,d})$ we denote the set of all fixed point of $f_{b,c,d}$.
We stress that the description of the set of fixed points of $f_{b,c,d}$ is important in finding of $p$-adic
Gibbs measures for $p$-adic Potts model \cite{ALS,MR1,M09}.

\begin{pro}
Let $p\geq3$ and $f_{b,c,d}$ be given by \eqref{f_{b,c,d}}. Then $Fix(f_{b,c,d})\subset\mathcal E_p$.
\end{pro}
\begin{proof}
Let us take $x\notin\mathbb Q_p\setminus\mathcal E_p$.
We consider two several cases: $|x|_p>1$ and $|x|_p\leq1$.\\
{\it Case $|x|_p>1$.}
Then we have
$$
\frac{bx-c}{x-d}-1=\frac{(b-1)x+d-c}{x-d}=\frac{o[1]x+o[1]}{O[x]+O[1]}=\frac{o[x]}{O[x]}=o[1],
$$
which yields $\frac{bx-c}{x-d}\in\mathcal E_p$. Then, due to Lemma \ref{epproperty} we infer $f_{b,c,d}(x)\in\mathcal E_p$ which means that
$x\notin Fix(f_{b,c,d})$.\\
{\it Case $|x|_p\leq1$} In this case, we can easily find $x-d=O[1]$. Noting that fact one has
$$
\frac{bx-c}{x-d}-1=\frac{(b-1)x+d-c}{x-d}=\frac{o[1]x+o[1]}{O[1]}=\frac{o[1]}{O[1]}=o[1].
$$
Again thanks to Lemma \ref{epproperty} we conclude that $x\notin Fix(f_{b,c,d})$.

Hence, we have shown that $Fix(f_{b,c,d})\subset\mathcal E_p$ which completes the proof.
\end{proof}

Now we are going to find all fixed points of $f_{b,c,d}$.

\begin{thm}
Let $p\geq3$ and $f_{b,c,d}$ be a given by \eqref{f_{b,c,d}}. Assume that $d=1-b+c$ and $|b-1|_p<|(c-1)^2|_p<|k^2|_p$.
Then $f_{b,c,d}$ has exactly $\kappa_p+1$ fixed points, where $\kappa_p=(k,p-1)$.
\end{thm}

\begin{proof}
We note that condition $d=1-b+c$ provides $Fix(f_{b,c,d})\neq\emptyset$, i.e. $x_{(0)}=1$ is a fixed point of $f_{b,c,d}$.
First we show the set of fixed points of $f_{b,c,d}$ coincides
with the set of fixed points of
$$
g_{b,c,d}(y)=\frac{by^k-c}{y^k-d}.
$$
Indeed, if $y\in Fix(g_{b,c,d})$ then one can check that $y^k\in Fix(f_{b,c,d})$. This means that
\begin{equation}\label{172}
card\left(Fix(g_{b,c,d})\right)\leq card\left(Fix(f_{b,c,d})\right).
\end{equation}
On the other hand, if $x\in Fix(f_{b,c,d})$, then for $z=\frac{bx-c}{x-d}$ we have
$$
g_{b,c,d}(z)=\frac{bz^k-c}{z^k-d}=\frac{bx-c}{x-d}=z.
$$
So, we infer that $z\in Fix(g_{b,c,d})$ which yields
\begin{equation}\label{171}
card\left(Fix(f_{b,c,d})\right)\leq card\left(Fix(g_{b,c,d})\right).
\end{equation}
From \eqref{172} and \eqref{171} we conclude that
$$
card\left(Fix(g_{b,c,d})\right)=card\left(Fix(f_{b,c,d})\right).
$$
It is clear that $y_{(0)}=1$ is a fixed point of $g_{b,c,d}$. Then from
$$
\frac{g_{b,c,d}(y)-1}{y-1}=1
$$
using $d=1-b+c$ we obtain
\begin{equation}\label{eq34}
y^k-c-(b-1)\sum_{j=1}^{k-1}y^{k-j}=0.
\end{equation}
Hence, since $|b-1|_p<|k^2|_p$ and $|c-1|_p<|k|_p$, the polynomial $$P(y)=y^k-c-(b-1)\sum_{j=1}^{k-1}y^{k-j}$$
satisfies all conditions of Theorem \ref{thm_poly}. So, we may conclude that the equation \eqref{eq34}
has exactly $\kappa_p$ solutions: $y_{(1)},y_{(2)},\dots,y_{(\kappa_p)}$. Consequently, $g_{b,c,d}$
has $\kappa_p+1$ fixed points: $y_{(0)}=1$ and $y_{(1)},y_{(2)},\dots,y_{(\kappa_p)}$. This means that
$f_{b,c,d}$ has exactly $\kappa_p+1$ fixed points: $x_{(i)}=y_{(i)}^k$, $i\in\{0,1,2,\dots,\kappa_p\}$.
This completes the proof.
\end{proof}

We notice that the obtained result allows to investigate dynamical behavior of the function $f_{b,c,d}$, for certain particular values of the parameters, chaoticity of such type of function has been investigated in \cite{ALS,MFKhO3}.

\section*{Acknowledgments}
The present work is supported by the UAEU "Start-Up" Grant, No.
31S259. The authors are grateful to anonymous referees whose useful remarks allowed to improve the presentation of the present paper.


\begin{thebibliography}{99}

\bibitem{ALB} Albeverio S., Cianci R., Khrennikov A.Yu., $p$-adic valued quantization, \textit{$p$-Adic Numbers,
Ultrametric Analysis, and Applications},  {\bf 1} (2009), 91--104.

\bibitem{ALS} Ahmad M.A.Kh., Liao L.M., Saburov M., Periodic $p$-adic Gibbs measures of $q$-state Potts model on Cayley
tree: the chaos implies the vastness of $p$-adic Gibbs measures, \textit{J. Stat. Phys.} {\bf 71} (2018), 1000--1034.

\bibitem{AK}  Anashin V., Khrennikov A., \textit{Applied algebraic dynamics}, Walter de Gruyter, Berlin, New
York, 2009.

\bibitem{AKY}  Anashin V.S., Khrennikov A.Yu.,  Yurova E.I., Characterization of
ergodicity of $p$-adic dynamical systems by using the van der Put basis, \textit{Dokl. Math.} {\bf 83} (2011),
306--308.

\bibitem{Borevich} Borevich Z.I., Shafarevich I.R., \textit{Number Theory},  Academic Press, New York, 1966.


\bibitem{COR} Casas J.M., Omirov B.A., Rozikov U.A., Solvability criteria for the equation $x^q = a$ in the field of $p$-adic numbers,
{\it Bull. Malays. Math. Sci. Soc.}, {\bf 37}(2014), 853-864.

\bibitem{Drag} Dragovich B., Khrennikov A.Yu., Kozyrev S.V., Volovich
I.V., On $p$-adic mathematical physics, \textit{p-Adic Numbers,
Ultrametric Analysis and Appl.} {\bf 1} (2009),  1–--17.

\bibitem{Drag2} Dragovich B., Khrennikov A.Yu., Kozyrev S.V., Volovich I.V., Zelenov E. I., $p$ -Adic
Mathematical Physics: The First 30 Years. \textit{$p$-Adic Numbers Ultrametric Anal. Appl.} {\bf 9} (2017), 87--121.

\bibitem{Khr} Khrennikov A.,  \textit{Non-Archimedean analysis and its applications.} Nauka, Fizmatlit, Moscow,
2003 (in Russian).

\bibitem{KKZ}  Khrennikov A.Yu., Kozyrev S.V.,  Zuniga-Galindo W.A., \textit{Ultrametric Pseudodifferential Equations
and Applications}, Cambridge Univ. Press, 2018.

\bibitem{Ko} Koblitz N.  \textit{$p$-adic numbers, $p$-adic analysis and
zeta-function}, Berlin, Springer, 1977.

\bibitem{M09}  Mukhamedov F., On existence of generalized Gibbs measures for one dimensional $p$-adic countable state Potts
model, \textit{Proc. Steklov Inst. Math.} {\bf 265} (2009), 165--176.

\bibitem{ReportMath} Mukhamedov F., A dynamical system approach to phase transitions for
$p$-adic Potts model on the Cayley tree of order two, {\it Rep. Math. Phys.} {\bf 70}(2012), 385--406.

\bibitem{MOSM13} Mukhamedov F., Omirov B., Saburov M.,  Masutova K.,  On solvability of cubic equations in the set of
$p$-adic integers ($p > 3$) \textit{Siberian Math. J.} {\bf 54}(2013), 501--516.

\bibitem{M14} Mukhamedov F., Recurrence equations over trees in a non-Archimedean context, \textit{P-adic Numb. Ultra. Anal.
Appl.} {\bf 6}(2014), 310--317.

\bibitem{MR1} Mukhamedov F.M., Rozikov U.A., On Gibbs measures of $p$-adic Potts model on a Cayley tree, \textit{Indag. Math.
N.S.} {\bf 15}(2004), 85--100.

\bibitem{FMMS_JNT} Mukhamedov F., Saburov M., On equation $x^q=a$ over $\mathbb Q_p$, {\it J. Number Theory},
{\bf 133} (2013), 55--58.

\bibitem{MFKhO3}  Mukhamedov F., Khakimov O., Chaotic behaviour of the
$p$-adic Potts-Bethe
mapping, \textit{Disc.  Cont. Dyn. Syst.} {\bf 38}(2018), 231--245.

\bibitem{Rosen} Rosen K.H., \textit{Elementary Number Theory and Its Applications}, Pearson, 2011.

\bibitem{sch} Schikhof W. H., \textit{Ultrametric calculus. An introduction to $p$-adic
analysis.} Cambridge: Cambridge University Press 1984.

\bibitem{VVZ} Vladimirov, V.S., Volovich, I.V., Zelenov, E.I. \textit{p-adic Analysis and Mathematical Physics},
World Scientific, Singapore, 1994.

\bibitem{YurKhren} Yurova E., Khrennikov A., Generalization of Hensel's lemma: finding the root
of $p$-adic Lipschitz functions, {\it  J. Number Theory} {\bf158} (2016), 217--233.

\bibitem{YurKhren1} Yurova Axelsson E.,  Khrennikov A. Subcoordinate representation of $ p$-adic functions and generalization of Hensel's lemma, \textit{Izv. Math.}  {\bf 82} (2018), 632--645.

\end{thebibliography}
\end{document}